\documentclass[11pt]{amsart}

\usepackage{evearticle}
\usepackage{stmaryrd}

\newcommand{\bes}{\mathrm{BES}}

\begin{document}

\title{The conformal dimension of the Brownian tree is one}
\author{Jason Miller and Yi Tian}

\date{\today}

\begin{abstract}
The Brownian tree, also known as the continuum random tree, is a canonical random compact, geodesic $\BR$-tree that arises as the universal scaling limit for numerous models of discrete random trees. A key quasisymmetric invariant of a metric space is its conformal dimension, defined as the infimum of the Hausdorff dimensions over all quasisymmetrically equivalent spaces. This value is always bounded below by the space's topological dimension and above by its Hausdorff dimension. In the present paper, we prove that the conformal dimension of the Brownian tree is $1$, matching its topological dimension.
\end{abstract}

\maketitle

\tableofcontents

\setlength{\parindent}{0pt}
\setlength{\parskip}{0.5\baselineskip plus 1pt minus 1pt}

\section{Introduction}

The \emph{Brownian tree} (a.k.a.~the \emph{continuum random tree}), introduced by Aldous \cite{MR1085326,MR1166406,MR1207226}, is a canonical random compact, geodesic $\BR$-tree that arises as the universal scaling limit for numerous models of discrete random trees, e.g., critical Galton--Watson trees, critical multi-type Galton--Watson trees \cite{MR2469338}, uniform unordered binary trees \cite{MR2829313}, uniform unordered trees \cite{MR3050512}, uniform unlabeled unrooted trees \cite{MR3983790}, critical random graphs \cite{abg2012criticalrandomgraphs}, and the minimum spanning tree \cite{adbgm2017mst}. It is also the scaling limit of random dissections \cite{MR3382675}, random planar maps with a unique large face \cite{MR3342658}, random outerplanar maps \cite{MR3573291}, and random graphs from subcritical classes \cite{MR3551197}. The Brownian tree also has connections to branching processes, superprocesses \cite{MR1714707}, and random geometry \cite{cs2004tightness,MR3112934,MR3070569,s2016hc,MR4007665,dms2021mating}.  There are many other examples in the literature where the Brownian tree arises that we have not listed here because they are far too numerous.

Let $(X, d_X)$ and $(Y, d_Y)$ be metric spaces. A homeomorphism $f \colon X \to Y$ is said to be \emph{quasisymmetric} if we can find an increasing homeomorphism $\eta \colon \BR_{\ge 0} \to \BR_{\ge 0}$ that bounds the distortion of relative distances:
\begin{equation*}
    \frac{d_Y(f(x), f(z))}{d_Y(f(y), f(z))} \le \eta\!\left(\frac{d_X(x, z)}{d_X(y, z)}\right)
\end{equation*}
for all triples of distinct points $x, y, z \in X$. The class of quasisymmetric mappings is closed under both inversion and composition. Consequently, the existence of such a mapping establishes an equivalence relation, and we say that $(X, d_X)$ and $(Y, d_Y)$ are \emph{quasisymmetrically equivalent}. This concept was originally formulated by \cite{MR595180} to extend the theory of conformal and quasiconformal mappings to arbitrary metric spaces.

The \emph{conformal dimension}, introduced by Pansu \cite{MR1024425}, is a quasisymmetric invariant that plays an important role in the classification of metric spaces up to quasisymmetric equivalence. It is defined as the infimal Hausdorff dimension among all spaces within a given quasisymmetric equivalence class. By definition, this dimension must lie between the topological dimension of the space and its original Hausdorff dimension.

The problem of computing or estimating conformal dimension has been studied for many general metric spaces and concrete deterministic fractal spaces \cite{MR1676353,MR1833245,MR2135168,MR2239342,MR2268118,MR2846307,MR3213834,MR3276005,MR3815187,MR4056527}. Since many naturally occurring fractals arise from random processes, it is very natural and interesting to study their conformal dimension. In \cite{MR4259151}, it was shown that the conformal dimension of fractal percolation is almost surely strictly smaller than its Hausdorff dimension, although explicit values for the conformal dimension have not yet been determined. The conformal dimension of the graph of a one-dimensional Brownian motion is almost surely equal to its Hausdorff dimension $3/2$ \cite{MR4912977}. The authors proved in a previous paper \cite{ConfDimBSph} that the conformal dimension of the Brownian sphere is almost surely equal to its topological dimension $2$.

The main result of the present paper is that the conformal dimension of the Brownian tree matches its topological dimension.

\begin{theorem}\label{thm:main}
    Almost surely, the Brownian tree has conformal dimension $1$.
\end{theorem}

A metric space whose conformal dimension is equal to its Hausdorff dimension is referred to as \emph{minimal}. In contrast to \cite{MR4912977}, to the best of our knowledge, the Brownian sphere \cite{ConfDimBSph} and the Brownian tree studied in the present paper are the first random fractals whose conformal dimensions are shown to be equal to their topological dimensions.

The \emph{quasisymmetric uniformization problem} asks for natural conditions under which a given metric space is quasisymmetrically equivalent to a model space. For results regarding standard model spaces, see \cite{MR595180} for $\BS^1$; \cite{MR1930885} alongside \cite{MR3608292,MR4073230,MR4608329,MR4956983} for $\BS^2$; and \cite{MR2854086} for round carpets. Several recent papers have explored these types of questions in the context of trees. For instance, \cite{bonktran2021csst} introduced the \emph{continuum self-similar tree (CSST)} as a deterministic model space for a compact metric $\BR$-tree, possessing the property that the Brownian tree is almost surely homeomorphic (but not quasisymmetrically equivalent) to it. Moreover, \cite{bonkmeyer2022uniformlybranching} proved that a \emph{quasiconformal tree} (a compact metric $\BR$-tree that is doubling and of bounded turning) is quasisymmetrically equivalent to the CSST if and only if it is trivalent and uniformly branching. Furthermore, \cite{MR4107537} proved that any quasiconformal tree is quasisymmetrically equivalent to a geodesic $\BR$-tree, and \cite{MR3660231} proved that its conformal dimension is $1$.

$\BR$-trees also play an important role in the study of Gromov hyperbolic groups. For example, by using the theory of group actions on $\BR$-trees and Rips' machine \cite{MR1346208}, it was proven in \cite{MR1638764} that if the boundary of a one-ended hyperbolic group contains a local cut point, the group splits over a two-ended subgroup.

We remark that the Brownian tree almost surely cannot be embedded quasisymmetrically into~$\BR^n$, or any doubling space \cite{MR4242630}. In particular, the Ahlfors regular conformal dimension and the conformal Assouad dimension (variants of the conformal dimension) of the Brownian tree are almost surely infinite. 

As in \cite{ConfDimBSph}, the arguments in the present paper are motivated by \cite{ConformalGauge}, where the authors constructed metrics quasisymmetrically equivalent to a prescribed one using the technique of hyperbolic fillings and admissible weight functions. Roughly speaking, an admissible weight function assigns a non-negative real number to each metric ball in the original metric space. This number corresponds to the ``relative size'' of the ball compared to its ``parent'' ball under the new metric associated with the weight function. See \Cref{subsubsection:hyperbolic-filling} for a short review and \cite[Section~5]{ConfDimBSph} for more detailed explanations.

We will construct an explicit admissible weight function for the hyperbolic fillings of the Brownian tree (cf.~\Cref{section:weight}). Intuitively, our admissible weight function produces a new metric that ``shrinks the size of the leaves and maintains the size of the skeleton of the tree''. Since the skeleton (the collection of points with degree at least $2$) almost surely has Hausdorff dimension $1$, our new metric decreases the Hausdorff dimension of the tree to near $1$.

\subsection*{Acknowledgements}

J.M.~received support from ERC consolidator grant ARPF (Horizon Europe UKRI G120614). Y.T.~was supported by a Cambridge International Scholarship from the Cambridge Trust. 

\section{Background}

\subsection{Notation}

The notation $\BR$ (resp.~$\BZ$; $\BN$) will be used to denote the set of real numbers (resp.~integers; positive integers). For $a < b$, we shall write $[a, b]_\BZ \defeq [a, b] \cap \BZ$.

Let $(X, D)$ be a metric space. Then we shall write $\diam(X; D) \defeq \sup\{D(x, y) : x, y \in X\}$. For $x \in X$ and $0 < s < t$, we shall write 
\begin{equation*}
    B_t(x; D) \defeq \{y \in X : D(x, y) < t\} \quad \text{and} \quad A_{s,t}(x; D) \defeq \{y \in X : s < D(x, y) < t\}. 
\end{equation*}

\subsection{The Brownian tree}\label{subsection:Brownian-tree}

In the present subsection, we recall the definitions and basic properties of the Brownian tree and its infinite counterpart.

Recall that a topological space $\SCT$ is called an \emph{$\BR$-tree} if, for every pair of points $x, y \in \SCT$, there exists a unique (up to reparameterization) continuous injective mapping $P \colon [0, 1] \to \SCT$ from $x$ to $y$. We shall denote the range of this mapping by $\llbracket x, y\rrbracket$. 

\subsubsection{The Brownian tree}\label{subsubsection:Brownian-tree}

Let $\{X_t\}_{t \in [0, 1]}$ be a normalized Brownian excursion. Define 
\begin{equation*}
    d(s, t) \defeq X_s + X_t - 2\inf_{r \in [s, t]} X_r, \quad \forall 0 \le s \le t \le 1. 
\end{equation*}
Note that $d$ is a pseudo-metric on $[0, 1]$. We define $\SCT \defeq [0, 1]/\sim$, where $s \sim t$ if and only if $d(s, t) = 0$, and let $d_\SCT$ be the induced metric on $\SCT$. Let $\mathop{\mathrm{proj}} \colon [0, 1] \to \SCT$ be the canonical projection mapping. We equip $(\SCT, d_\SCT)$ with the pushforward measure $\SM_\SCT$ of the Lebesgue measure on $[0, 1]$, along with the root $o \defeq \mathop{\mathrm{proj}}(0) = \mathop{\mathrm{proj}}(1)$. The quadruple $(\SCT, d_\SCT, \SM_\SCT; o)$ is called the \emph{Brownian tree} with unit volume.

It is a standard fact that $(\SCT, d_\SCT)$ is almost surely a geodesic $\BR$-tree with Hausdorff dimension equal to $2$. Moreover, conditional on $(\SCT, d_\SCT, \SM_\SCT)$, the root $o$ is distributed according to the volume measure $\SM_\SCT$. 

We shall write $\scn$ for the Brownian excursion measure normalized so that
\begin{equation}\label{eq:Brownian-excursion-measure}
    \scn\!\left(\sup_t e_t > r\right) = \frac1{2r}, \quad \forall r > 0. 
\end{equation}
The Brownian tree measure $\SCN$ is defined via the same procedure described above by replacing the law of the normalized Brownian excursion with the measure $\scn$. Equivalently, $\SCN$ is given by the pushforward of the measure $\BP \otimes \frac1{2\sqrt{2\pi}} a^{-3/2} \, \rd a$ under the scaling $(\SCT, d_\SCT, \SM_\SCT; o) \mapsto (\SCT, a^{1/2} d_\SCT, a \SM_\SCT; o)$, where $\BP$ denotes the law of the Brownian tree with unit volume.

\subsubsection{The infinite Brownian tree}\label{subsubsection:infinite-Brownian-tree}

The infinite Brownian tree is an unbounded variant of the Brownian tree and is constructed in a similar manner. More precisely, let $R$ and $\widetilde R$ be two independent Bessel processes of dimension $3$ ($\bes^3$) starting from $0$. We set $X_t \defeq R_t$ for $t \ge 0$, and $X_t \defeq \widetilde R_{-t}$ for $t < 0$. Define
\begin{equation*}
    d(s, t) \defeq 
    \begin{cases}
        X_s + X_t - 2\inf_{r \in [s, t]} X_r & \text{if } s \le t \le 0 \text{ or } 0 \le s \le t; \\
        X_s + X_t - 2\inf_{r \in (-\infty, s] \cup [t, \infty)} X_r & \text{if } s \le 0 \le t. 
    \end{cases}
\end{equation*}
We define $\SCT \defeq \BR/\sim$, where $s \sim t$ if and only if $d(s, t) = 0$, and let $d_\SCT$ be the induced metric on $\SCT$. Let $\mathop{\mathrm{proj}} \colon \BR \to \SCT$ be the canonical projection mapping. We equip $(\SCT, d_\SCT)$ with the pushforward measure $\SM_\SCT$ of the Lebesgue measure on $\BR$, along with the root $o \defeq \mathop{\mathrm{proj}}(0)$. The quadruple $(\SCT, d_\SCT, \SM_\SCT; o)$ is called the \emph{infinite Brownian tree}.

The infinite Brownian tree $(\SCT, d_\SCT, \SM_\SCT; o)$ satisfies the scaling property that for each deterministic $\lambda > 0$, $(\SCT, \lambda d_\SCT, \lambda^2\SM_\SCT; o)$ has the same law as $(\SCT, d_\SCT, \SM_\SCT; o)$, and the re-rooting property that for each deterministic $t \in \BR$, $(\SCT, d_\SCT, \SM_\SCT; \mathop{\mathrm{proj}}(t))$ has the same law as $(\SCT, d_\SCT, \SM_\SCT; o)$. The latter follows from the fact that for each deterministic $t > 0$, if we write for each $s \ge 0$, 
\begin{equation}\label{eq:re-rooting}
    \begin{dcases}
        R_s^{(t)} \defeq R_t + R_{t + s} - 2\inf_{r \in [t, t + s]} R_r; \\
        \widetilde R_s^{(t)} \defeq 
        \begin{cases}
            R_t + R_{t - s} - 2\inf_{r \in [t - s, t]} R_r & \text{if } s \le t; \\
            R_t + \widetilde R_{s - t} - 2\left(\inf_{r \in [t, \infty)} R_r \wedge \inf_{r \in [s - t, \infty)} \widetilde R_r\right) & \text{if } s > t, 
        \end{cases}
    \end{dcases}
\end{equation}
then $(R, \widetilde R)$ and $(R^{(t)}, \widetilde R^{(t)})$ have the same law.

The infinite Brownian tree can also be constructed as follows. We start with an infinite spine identified with $[0, \infty)$. Let $\{(\SCT_j, t_j)\}_j$ be a Poisson point process (PPP) with intensity measure $2\SCN \otimes \rd t$. The infinite Brownian tree is then obtained by attaching the root of each tree $\SCT_j$ to the infinite spine at the corresponding point $t_j$.

Let $(\SCT, d_\SCT, \SM_\SCT; o)$ be an infinite Brownian tree. For each $x \in \SCT$, there exists a unique geodesic ray $P \colon [0, \infty) \to \SCT$ with $P(0) = x$, whose range we denote by $\llbracket x, \infty\rrparenthesis$. For any pair $x, y \in \SCT$, we define $x \curlywedge y$ as the unique point $z \in \SCT$ satisfying $\llbracket x, \infty\rrparenthesis \cap \llbracket y, \infty\rrparenthesis = \llbracket z, \infty\rrparenthesis$. Additionally, for $t \ge 0$, let $\SR_t(x)$ be the unique point on $\llbracket x, \infty\rrparenthesis$ such that $d_\SCT(x, \SR_t(x)) = t$. Note that $\{\SR_t\}_{t \ge 0}$ forms a semigroup, i.e., $\SR_s(\SR_t(x)) = \SR_{s + t}(x)$ for all $x \in \SCT$ and $s, t \ge 0$. Moreover, the closure of each connected component of $\SCT \setminus \llbracket x, \infty\rrparenthesis$ is naturally a rooted (geodesic) $\BR$-tree. We shall refer to these $\BR$-trees as subtrees branching off $\llbracket x, \infty\rrparenthesis$.

\begin{lemma}\label{lem:metric-ball-volume}
    Let $(\SCT, d_\SCT, \SM_\SCT; o)$ be an infinite Brownian tree. Then for each $\zeta > 0$, there almost surely exists $\varepsilon_\ast \in (0, 1)$ such that 
    \begin{equation}\label{eq:metric-ball-volume}
        \SM_\SCT(B_\varepsilon(z; d_\SCT)) \ge \varepsilon^{2 + \zeta}, \quad \forall z \in B_1(o; d_\SCT), \ \forall \varepsilon \in (0, \varepsilon_\ast]. 
    \end{equation}
\end{lemma}

\begin{proof}
    Let $\mathop{\mathrm{proj}} \colon \BR \to \SCT$ be the canonical projection mapping. Since there almost surely exists $T > 0$ such that $B_1(o; d_\SCT) \subset \mathop{\mathrm{proj}}([-T, T])$, it suffices to show that for each $T > 0$, there almost surely exists $\varepsilon_\ast \in (0, 1)$ such that~\eqref{eq:metric-ball-volume} holds with $\mathop{\mathrm{proj}}([-T, T])$ in place of $B_1(o; d_\SCT)$. It follows from the discussion surrounding~\eqref{eq:re-rooting} that $(\SCT, d_\SCT, \SM_\SCT; o, \mathop{\mathrm{proj}}([-T, T]))$ and $(\SCT, d_\SCT, \SM_\SCT; \mathop{\mathrm{proj}}(2T), \mathop{\mathrm{proj}}([T, 3T]))$ have the same law. Thus, it suffices to show that there almost surely exists $\varepsilon_\ast \in (0, 1)$ such that~\eqref{eq:metric-ball-volume} holds with $\mathop{\mathrm{proj}}([T, 3T])$ in place of $B_1(o; d_\SCT)$. This follows immediately from the definitions of $d_\SCT$ and $\SM_\SCT$, together with the fact that the paths of a $\bes^3$ process are almost surely locally $\gamma$-H\"older continuous of every order $\gamma \in (0, 1/2)$.
\end{proof}

\begin{lemma}\label{lem:net}
    Let $(\SCT, d_\SCT, \SM_\SCT; o)$ be an infinite Brownian tree. Let $\mathop{\mathrm{proj}} \colon \BR \to \SCT$ be the canonical projection mapping. Let $T > 0$. Given $\SCT$ and $\mathop{\mathrm{proj}}$, let $\{x_n\}_{n \in \BN}$ be conditionally independent samples from $\SM_\Phi|_{\mathop{\mathrm{proj}}([-T, T])}$ (renormalized to be a probability measure). Then for each $\zeta > 0$, almost surely on the event that $B_1(o; d_\SCT) \subset \mathop{\mathrm{proj}}([-T, T])$, there exists $\varepsilon_\ast \in (0, 1)$ such that 
    \begin{equation}\label{eq:net}
        B_1(o; d_\SCT) \subset \bigcup \left\{B_\varepsilon(x_n; d_\SCT) : n \in [1, \varepsilon^{-2 - \zeta}]_\BZ, \ x_n \in B_1(o; d_\SCT)\right\}, \quad \forall \varepsilon \in (0, \varepsilon_\ast].
    \end{equation}
\end{lemma}

\begin{proof}
    For $\varepsilon_0 > 0$, write $E_{T,\varepsilon_0}$ for the event that $B_1(o; d_\SCT) \subset \mathop{\mathrm{proj}}([-T, T])$ and that
    \begin{equation*}
        \SM_\SCT(B_\varepsilon(z; d_\SCT)) \ge \varepsilon^{2 + \zeta/2}, \quad \forall z \in B_1(o; d_\SCT), \ \forall \varepsilon \in (0, \varepsilon_0]. 
    \end{equation*}
    By \Cref{lem:metric-ball-volume}, it suffices to show that, almost surely on the event $E_{T,\varepsilon_0}$, there exists $\varepsilon_\ast \in (0, 1)$ such that~\eqref{eq:net} holds. For $j \in \BN$, write $F_j$ for the event that
    \begin{equation*}
        B_1(o; d_\SCT) \subset \bigcup \left\{B_{2^{-j - 1}}(x_n; d_\SCT) : n \in [1, 2^{j(2 + \zeta)}]_\BZ, \ x_n \in B_1(o; d_\SCT)\right\}. 
    \end{equation*}
    We observe that if $F_j$ occurs for all $j \ge \lfloor\log_2(1/\varepsilon_\ast)\rfloor$, then~\eqref{eq:net} holds. Thus, it suffices to show that, almost surely on the event $E_{T,\varepsilon_0}$, there exists $j_\ast \in \BN$ such that $F_j$ occurs for all $j \ge j_\ast$. By definition, if $F_j$ does not occur, then there exists $x \in B_1(o; d_\SCT)$ such that $d_\SCT(x, x_n) \ge 2^{-j - 1}$ for all $n \in [1, 2^{j(2 + \zeta)}]_\BZ$ with $x_n \in B_1(o; d_\SCT)$, in which case one verifies immediately that there exists $x^\prime \in \SCT$ such that $B_{2^{-j - 2}}(x^\prime; d_\SCT) \subset B_1(o; d_\SCT)$ and that $x_n \notin B_{2^{-j - 2}}(x^\prime; d_\SCT)$ for all $n \in [1, 2^{j(2 + \zeta)}]_\BZ$ with $x_n \in B_1(o; d_\SCT)$ (hence for all $n \in [1, 2^{j(2 + \zeta)}]_\BZ$). Write $G_j$ for the event that $B_{2^{-j - 3}}(x_{\lfloor2^{j(2 + \zeta)}\rfloor + 1}; d_\SCT) \subset B_1(o; d_\SCT)$ and that $x_n \notin B_{2^{-j - 3}}(x_{\lfloor2^{j(2 + \zeta)}\rfloor + 1}; d_\SCT)$ for all $n \in [1, 2^{j(2 + \zeta)}]_\BZ$. Then
    \begin{equation*}
        \BP\lbrack G_j \mid E_{T,\varepsilon_0} \cap F_j^c\rbrack \ge \BP\!\left\lbrack x_{\lfloor2^{j(2 + \zeta)}\rfloor + 1} \in B_{2^{-j - 3}}(x^\prime; d_\SCT) \ \middle\vert \ E_{T,\varepsilon_0} \cap F_j^c\right\rbrack \ge 2^{-(j + 3)(2 + \zeta/2)}/(2T). 
    \end{equation*}
    On the other hand, 
    \begin{equation*}
        \BP\lbrack G_j \mid E_{T,\varepsilon_0}\rbrack \le (1 - 2^{-(j + 3)(2 + \zeta/2)}/(2T))^{\lfloor2^{j(2 + \zeta)}\rfloor} \le \exp(-c2^{j\zeta/2})
    \end{equation*}
    for some $c = c(\zeta, T) > 0$. Thus, we conclude that
    \begin{equation*}
        \BP\lbrack F_j^c \mid E_{T,\varepsilon_0}\rbrack \le \frac{\BP\lbrack G_j \mid E_{T,\varepsilon_0}\rbrack}{\BP\lbrack G_j \mid E_{T,\varepsilon_0} \cap F_j^c\rbrack} \le \exp(-c2^{j\zeta/2}) \cdot 2^{(j + 3)(2 + \zeta/2)} \cdot (2T). 
    \end{equation*}
    Combining this with the Borel--Cantelli lemma, we complete the proof.
\end{proof}

\subsection{Gromov hyperbolic geometry and hyperbolic fillings}

In the present subsection, we review standard material on Gromov hyperbolic geometry and the technique of hyperbolic fillings. This technique constructs a Gromov hyperbolic space with a prescribed Gromov boundary. By adapting the arguments of \cite{ConformalGauge}, we associate to each admissible weight function on the hyperbolic filling a metric space quasisymmetrically equivalent to the prescribed one. See \cite[Section~5]{ConfDimBSph} for more details.

\subsubsection{Gromov hyperbolic geometry}

Let $(X, d)$ be a proper and geodesic metric space. Then $(X, d)$ is referred to as \emph{Gromov hyperbolic} if there exists $\delta \ge 0$ such that $P_{x,y} \subset B_\delta(P_{x,z} \cup P_{y,z}; d)$ for all $x, y, z \in X$, where $P_{x,y}$ denotes any $d$-geodesic connecting $x$ and $y$.

Suppose $(X, d)$ is a Gromov hyperbolic space equipped with a basepoint $o \in X$. The \emph{Gromov product} of $x, y \in X$ with respect to $o$ is defined as $(x, y)_o \defeq \frac12(d(x, o) + d(y, o) - d(x, y))$. One can construct the \emph{Gromov boundary} $\partial_\infty(X, d)$ by
\begin{equation*}
    \partial_\infty(X, d) \defeq \{\{x_j\}_{j \in \BN} \subset X : (x_j, x_k)_o \to \infty \text{ as } j, k \to \infty\}/\sim, 
\end{equation*}
where two sequences $\{x_j\}_{j \in \BN}$ and $\{y_j\}_{j \in \BN}$ are equivalent if $(x_j, y_j)_o \to \infty$ as $j \to \infty$. This boundary is naturally identified with the set of $d$-geodesic rays $P \colon [0, \infty) \to X$, modulo an equivalence relation which identifies any two rays whose ranges are within a finite $d$-Hausdorff distance. We can extend the Gromov product to any pair of boundary points $a, b \in \partial_\infty(X, d)$ by setting
\begin{equation*}
    (a, b)_o \defeq \inf_{\{x_j\}_{j \in \BN}, \{y_j\}_{j \in \BN}} \liminf_{j \to \infty} (x_j, y_j)_o \in [0, \infty], 
\end{equation*}
minimizing over all valid representative sequences $\{x_j\}_{j \in \BN} \sim a$ and $\{y_j\}_{j \in \BN} \sim b$.

Finally, a metric $D$ on $\partial_\infty(X, d)$ is termed \emph{visual} with parameter $\varepsilon > 0$ if there is a constant $C \ge 1$ such that
\begin{equation*}
    C^{-1} \re^{-\varepsilon(a, b)_o} \le D(a, b) \le C \re^{-\varepsilon(a, b)_o}, \quad \forall a, b \in \partial_\infty(X, d). 
\end{equation*}
Such visual metrics are guaranteed to exist for all sufficiently small parameters $\varepsilon > 0$, and any two visual metrics belong to the same quasisymmetric equivalence class (i.e., the identity mapping on the Gromov boundary is a quasisymmetric homeomorphism between any pair of visual metrics).

\subsubsection{Hyperbolic fillings}\label{subsubsection:hyperbolic-filling}

Let $(X, D)$ be a compact metric space containing a point $x_0 \in X$ such that $X = B_1(x_0; D)$. 

Fix a sufficiently small parameter $\alpha \in (0, 1)$. Let $\{x_0\} = A_0 \subset A_1 \subset A_2 \subset \cdots$ be an increasing sequence of finite subsets of $X$ such that each $A_n$ is a maximal $\alpha^n$-separated subset, i.e., $D(x, y) \ge \alpha^n$ for all pairs of distinct $x, y \in A_n$ and $\{B_{\alpha^n}(x; D) : x \in A_n\}$ forms a covering of $X$. We shall write $V_n \defeq \{(x, n) : x \in A_n\}$ and $V \defeq \coprod_{n \ge 0} V_n$. For distinct $(x, m), (y, n) \in V$, we shall write $(x, m) \sim (y, n)$ if 
\begin{itemize}
    \item $m = n$ and $B_{4\alpha^m}(x; D) \cap B_{4\alpha^n}(y; D) \neq \emptyset$, or
    \item $\lvert m - n\rvert = 1$ and $B_{\alpha^m}(x; D) \cap B_{\alpha^n}(y; D) \neq \emptyset$. 
\end{itemize}

The metric graph associated with $(V, \sim)$ is Gromov hyperbolic with its Gromov boundary naturally identified with $X$. Moreover, $D$ induces a visual metric on $X$ with parameter $\log(1/\alpha)$. 

Fix an assignment $\sigma \colon X \times \BN \to \BR_{\ge 0}$. Suppose that the following condition is satisfied: 
\begin{equation}\label{eq:admissibility}
    \tag{$\bigstar$}
    \parbox{.85\linewidth}{For each $n \in \BN$ and $y, x_0, x_1, \ldots, x_N \in X$ such that $B_{4\alpha^n}(x_{j - 1}; D) \cap B_{4\alpha^n}(x_j; D) \neq \emptyset$ for all $j \in [1, N]_\BZ$, $B_{\alpha^{n - 1}}(y; D) \cap B_{4\alpha^n}(x_0; D) \neq \emptyset$, and $(X \setminus B_{2\alpha^{n - 1}}(y; D)) \cap B_{4\alpha^n}(x_N; D) \neq \emptyset$, we have $\sum_{j = 0}^N \sigma(x_j, n) \ge 1$.} 
\end{equation}

\begin{lemma}\label{lem:core}
    For each $\eta > 0$, there exists a metric $D_\sigma$ on $X$ and a constant $C > 0$ such that the identity mapping $\mathrm{id} \colon (X, D) \to (X, D_\sigma)$ is quasisymmetric and
    \begin{equation*}
        \diam(B_{\alpha^n}(x; D); D_\sigma) \le C \prod_{j = 1}^n \left(\eta + 2\sup\{\sigma(y, j) : y \in X \text{ with } D(x, y) \le 26\alpha^j\}\right), \quad \forall (x, n) \in V. 
    \end{equation*}
\end{lemma}

\begin{proof}
    We use the notation of \cite[Section~5]{ConfDimBSph}. Recall that for $(x, n) \in V$,
    \begin{itemize}
        \item $\nu(x, n) = 2 \sup\{\sigma(x^\pprime, n) : (x^\prime, n), (x^\pprime, n) \in V \text{ with } (x, n) \cong (x^\prime, n) \cong (x^\pprime, n)\}$ (where ``$\cong$'' denotes either ``$\sim$'' or ``$=$''); 
        \item $\mu(x, n) = \eta \vee \nu(x, n) \wedge (1 - \eta)$; 
        \item $g(x, n)_{n - 1} = (y, n - 1)$, where $y \in A_{n - 1}$ is such that $D(x, y) \le D(x, y^\prime)$ for all $y^\prime \in A_n \setminus \{y\}$ (hence that $D(x, y) \le \alpha^{n - 1}$); 
        \item $g(x, n)_j = g(\cdots g(g(x, n)_{n - 1})_{n - 2} \cdots)_j$;
        \item $\varrho \colon V \setminus \{(x_0, 0)\} \to [\eta, 1 - \eta]$ is an assignment such that $\mu(x, n) \le \varrho(x, n) \le \sup\{\sigma(x^\prime, n) : (x^\prime, n) \in V  \text{ with } (x, n) \cong (x^\prime, n)\}$ (cf.~\cite[Lemma~5.3, (i)]{ConfDimBSph}); 
        \item $\pi(x, n) = \prod_{j = 1}^n \varrho(g(x, n)_j)$. 
    \end{itemize}
    
    By the discussion of \cite[Section~5]{ConfDimBSph} (especially \cite[Corollary~5.6]{ConfDimBSph}), there is a natural quasisymmetric homeomorphism $(X, D) \to (\partial_\infty(Z, d_\varrho), d_1)$ (induced by a \emph{quasi-isometry} from the metric graph associated with $(V, \sim)$ to another Gromov hyperbolic space $(Z, d_\varrho)$) such that if we write $D_\sigma$ for the metric on $X$ induced by $d_1$, then there exists a constant $C > 0$ such that 
    \begin{equation*}
        \diam(B_{\alpha^n}(x; D); D_\sigma) \le C\pi(x, n) = C\prod_{j = 1}^n \varrho(g(x, n)_j), \quad \forall (x, n) \in V. 
    \end{equation*}
    On the other hand, it follows from the preceding discussion that
    \begin{align*}
        \varrho(g(x, n)_j) &\le \sup\{\mu(y, j) : (y, j) \in V \text{ with } g(x, n)_j \cong (y, j)\} \\
        &\le \eta + \sup\{\nu(y, j) : (y, j) \in V \text{ with } g(x, n)_j \cong (y, j)\} \\
        &\le \eta + 2\sup\{\sigma(y^\pprime, j) : (y, j), (y^\prime, j), (y^\pprime, j) \in V \text{ with } g(x, n)_j \cong (y, j) \cong (y^\prime, j) \cong (y^\pprime, j)\} \\
        &\le \eta + 2\sup\{\sigma(y, j) : y \in X \text{ with } D(x_j, y) \le 24\alpha^j\} \quad \text{(where } g(x, n)_j = (x_j, j)\text{)} \\
        &\le \eta + 2\sup\{\sigma(y, j) : y \in X \text{ with } D(x, y) \le 26\alpha^j\} \quad \text{(where } D(x_j, x) \le 2\alpha^j\text{)}.
    \end{align*}
    This completes the proof. 
\end{proof}

In particular, in order to show that the conformal dimension of $(X, D)$ is at most $p$ for some $p > 0$, it suffices to construct maximal $\alpha^n$-separated subsets $\{x_0\} = A_0 \subset A_1 \subset A_2 \subset \cdots \subset X$ and an admissible (i.e., satisfying~\eqref{eq:admissibility}) weight function $\sigma \colon X \times \BN \to \BR_{\ge 0}$ such that 
\begin{equation*}
    \sum_{(x, n) \in V_n} \prod_{j = 1}^n \left(\eta + 2\sup\{\sigma(y, j) : y \in X \text{ with } D(x, y) \le 26\alpha^j\}\right)^p \to 0 \quad \text{as } n \to \infty. 
\end{equation*}

\section{Constructing an admissible weight}\label{section:weight}

Let $(\SCT, d_\SCT, \SM_\SCT; o)$ be an infinite Brownian tree. In the present section, we construct a weight function $\sigma \colon \SCT \times \BN \to \BR_{\ge 0}$ that satisfies~\eqref{eq:admissibility}. Fix a sufficiently small parameter $\alpha \in (0, 1)$ to be chosen later. 

\begin{definition}\label{def:weight}
    Let $x \in \SCT$ and $n \in \BN$. Then:
    \begin{itemize}
        \item We shall write $E(x, n)$ (resp.~$\widetilde E(x, n)$) for the event that there exists a subtree branching off $\llbracket x, \SR_{4\alpha^n}(x)\rrbracket$ (resp.~$\llbracket x, \SR_{32\alpha^n}(x)\rrbracket$) with diameter at least $\alpha^{n - 1}/4$. (Recall from \Cref{subsubsection:infinite-Brownian-tree} that $\SR_{4\alpha^n}(x)$ denotes the unique point on the geodesic ray $\llbracket x, \infty\rrparenthesis$ such that $d_\SCT(x, \SR_{4\alpha^n}(x)) = 4\alpha^n$ and that a subtree branching off $\llbracket x, \SR_{4\alpha^n}(x)\rrbracket$ is the closure of a bounded connected component of $\SCT \setminus \llbracket x, \SR_{4\alpha^n}(x)\rrbracket$.)
        \item We shall write $\sigma(x, n) \defeq 32\alpha \cdot \mathbf 1_{E(x, n)}$. 
    \end{itemize}
\end{definition}

The point of introducing the event $\widetilde E(x, n)$ is to obtain a version of $E(x, n)$ that is robust under small perturbations of the center $x$.

\begin{lemma}
\label{lem:tilde-E-2-E}
    Let $x \in \SCT$ and $n \in \BN$. If $E(x^\prime, n)$ occurs for some $x^\prime \in B_{28\alpha^n}(x; d_\SCT)$, then $\widetilde E(x, n)$ occurs. 
\end{lemma}
\begin{proof}
    Suppose that $x^\prime \in B_{28\alpha^n}(x; d_\SCT)$ and $E(x^\prime, n)$ occurs. Then there exists a subtree $(\SCU^\prime, o^\prime)$ branching off $\llbracket x^\prime, \SR_{4\alpha^n}(x^\prime)\rrbracket$ with diameter at least $\alpha^{n - 1}/4$. If $o^\prime \in \llbracket x, \infty\rrparenthesis$, then $(\SCU^\prime, o^\prime)$ is a subtree branching off $\llbracket x, \infty\rrparenthesis$ such that $d_\SCT(x, o^\prime) \le d_\SCT(x, x^\prime) + d_\SCT(x^\prime, o^\prime) \le 28\alpha^n + 4\alpha^n = 32\alpha^n$, hence a subtree branching off $\llbracket x, \SR_{32\alpha^n}(x)\rrbracket$. Otherwise, we have $o^\prime \in \llbracket x^\prime, x \curlywedge x^\prime\rrbracket$, in which case the subtree branching off $\llbracket x, \infty\rrparenthesis$ at $x \curlywedge x^\prime$ contains $\SCU^\prime$ (hence has diameter at least $\alpha^{n - 1}/4$) and satisfies $d_\SCT(x, x \curlywedge x^\prime) \le d_\SCT(x, x^\prime) \le 28\alpha^n \le 32\alpha^n$. This completes the proof. 
\end{proof}

\begin{lemma}
\label{lem:admissibility}
    Condition~\eqref{eq:admissibility} is satisfied: Let $n \in \BN$ and $y, x_0, x_1, \ldots, x_N \in \SCT$ such that
    \begin{itemize}
        \item $B_{\alpha^{n - 1}}(y; d_\SCT) \cap B_{4\alpha^n}(x_0; d_\SCT) \neq \emptyset$,
        \item $B_{4\alpha^n}(x_{j - 1}; d_\SCT) \cap B_{4\alpha^n}(x_j; d_\SCT) \neq \emptyset$ for all $j \in [1, N]_\BZ$, and 
        \item $(\SCT \setminus B_{2\alpha^{n - 1}}(y; d_\SCT)) \cap B_{4\alpha^n}(x_N; d_\SCT) \neq \emptyset$. 
    \end{itemize}
    Then $\sum_{j = 0}^N \sigma(x_j, n) \ge 1$. 
\end{lemma}

\begin{proof}
    By possibly replacing the path $B_{4\alpha^n}(x_0; d_\SCT), B_{4\alpha^n}(x_1; d_\SCT), \ldots, B_{4\alpha^n}(x_N; d_\SCT)$ with a subpath, we may assume without loss of generality that 
    \begin{equation*}
        \left(\bigcup_{k = 0}^{j - 1} B_{4\alpha^n}(x_k; d_\SCT)\right) \cap \left(\bigcup_{k = j + 1}^N B_{4\alpha^n}(x_k; d_\SCT)\right) = \emptyset, \quad \forall j \in [1, N - 1]_\BZ. 
    \end{equation*}
    Note that for each $j \in [0, N]_\BZ$ such that $B_{4\alpha^n}(x_j; d_\SCT) \cap A_{4\alpha^{n - 1}/3,5\alpha^{n - 1}/3}(y; d_\SCT) \neq \emptyset$, the event $E(x, n)$ occurs. Indeed, 
    \begin{equation*}
        \bigcup_{k = 0}^{j - 2} B_{4\alpha^n}(x_k; d_\SCT) \quad \text{and} \quad \bigcup_{k = j + 2}^N B_{4\alpha^n}(x_k; d_\SCT)
    \end{equation*}
    are contained in distinct connected components of $\SCT \setminus B_{4\alpha^n}(x_j; d_\SCT)$ (hence at least one of them is contained in a subtree branching off $\llbracket x_j, \SR_{4\alpha^n}(x_j)\rrbracket$), and both of them are connected and have diameter at least $\alpha^{n - 1}/4$.

    Since the number of $j \in [0, N]_\BZ$ such that $B_{4\alpha^n}(x_j; d_\SCT) \cap A_{4\alpha^{n - 1}/3,5\alpha^{n - 1}/3}(y; d_\SCT) \neq \emptyset$ is at least $(\alpha^{n - 1}/3)/(8\alpha^n) = 1/(24\alpha)$, we conclude that $\sum_{j = 0}^N \sigma(x_j, n) \ge 32\alpha/(24\alpha) \ge 1$. This completes the proof. 
\end{proof}

\begin{definition}\label{def:rho-sigma-pi}
    Let $x \in \SCT$ and $n \in \BN$. Then we shall write 
    \begin{itemize}
        \item $\varrho(x, n) \defeq \eta + 2\sup\{\sigma(y, n) : y \in \SCT \text{ with } d_\SCT(x, y) \le 26\alpha^n\}$ (cf.~\Cref{lem:core});
        \item $\varsigma(x, n) \defeq \eta + 64\alpha \cdot \mathbf 1_{\widetilde E(x, n)}$; 
        \item $\varpi(x, n) \defeq \prod_{j = 1}^n \varsigma(x, j)$. 
    \end{itemize}
\end{definition}

\begin{lemma}\label{lem:rho-upper-bound}
    Let $x \in \SCT$ and $n \in \BN$. Then $\varrho(x, n) \le \varsigma(x^\prime, n)$ (hence $\prod_{j = 1}^n \varrho(x, j) \le \varpi(x^\prime, n)$) for all $x^\prime \in B_{2\alpha^n}(x; d_\SCT)$.
\end{lemma}

\begin{proof}
    This follows immediately from the various definitions involved, together with \Cref{lem:tilde-E-2-E}.
\end{proof}

\begin{lemma}\label{lem:pi-expectation}
    Set $\eta \defeq \alpha^{100}$. Then there is a universal constant $C > 0$ such that for each $p \in (1, 2)$, 
    \begin{equation*}
        \BE\lbrack\varpi(o, n)^p\rbrack \le C^n \alpha^{(p + 1)n}, \quad \forall \alpha \in (0, 1), \ \forall n \in \BN. 
    \end{equation*}
\end{lemma}

Before proceeding, we observe from the PPP construction of $\SCT$ (cf.~\Cref{subsubsection:infinite-Brownian-tree}) that for each $0 \le s < t$ and $r > 0$, the probability that there exists a subtree branching off $\llbracket\SR_s(o), \SR_t(o)\rrbracket$ with diameter at least $r$ is at most 
\begin{equation}\label{eq:pi-expectation-proof}
    1 - \exp\!\left(- 2 \cdot (t - s) \cdot \scn\!\left(\sup_t e_t > r/2\right)\right) = 1 - \re^{-2(t - s)/r} \le 2(t - s)/r. 
\end{equation}

\begin{proof}[Proof of \Cref{lem:pi-expectation}]
    Fix $p \in (1, 2)$. For each sequence of integers $\Bc = (n + 1 \ge b_0 > a_1 \ge b_1 > \cdots > a_\bullet \ge b_\bullet \ge 1)$ with $\bullet \in \BN \cup \{0\}$, consider the event $F_n(\Bc)$ that the following are true:
    \begin{itemize}
        \item If $b_0 \le n$, then there exists a subtree branching off $\llbracket o, \SR_{32\alpha^{n + 1}}(o)\rrbracket$ with diameter at least $\alpha^{b_0 - 1}/4$ (write $G_n(b_0)$ for this event), in which case the event $\widetilde E(o, j)$ occurs for all $j \in [b_0, n]_\BZ$. 
        \item For each $k \in [1, \bullet]_\BZ$, there exists a subtree branching off $\llbracket\SR_{32\alpha^{a_k + 1}}(o), \SR_{32\alpha^{a_k}}(o)\rrbracket$ with diameter at least $\alpha^{b_k - 1}/4$ (write $G(a_k, b_k)$ for this event), in which case the event $\widetilde E(o, j)$ occurs for all $j \in [b_k, a_k]_\BZ$. 
        \item For each $j \in [1, n]_\BZ \setminus \left([b_0, n]_\BZ \cup \bigcup_{k \in [1, \bullet]_\BZ} [b_k, a_k]_\BZ\right)$, the event $\widetilde E(o, j)$ does not occur. 
    \end{itemize}
    One verifies immediately that the union of the events $F_n(\Bc)$ for all possible sequences $\Bc$ is equal to the whole probability space. Write
    \begin{equation*}
        \len(\Bc) \defeq (n + 1 - b_0) + (a_1 - b_1 + 1) + \cdots + (a_\bullet - b_\bullet + 1) = \#\!\left([b_0, n]_\BZ \cup \bigcup_{k = 1}^\bullet [b_k, a_k]_\BZ\right) \le n.
    \end{equation*}
    On the event $F_n(\Bc)$, we have 
    \begin{equation*}
        \varpi(o, n) = (\eta + 64\alpha)^{\len(\Bc)} \cdot \eta^{n - \len(\Bc)} \le 65^n \cdot \alpha^{100n - 99\len(\Bc)}. 
    \end{equation*}
    Since the segments $\llbracket o, \SR_{32\alpha^{n + 1}}(o)\rrbracket$ and $\llbracket\SR_{32\alpha^{a_k + 1}}(o), \SR_{32\alpha^{a_k}}(o)\rrbracket$ for $k \in [1, \bullet]_\BZ$ are disjoint, it follows that the events $G_n(b_0), G(a_1, b_1), \ldots, G(a_\bullet, b_\bullet)$ are independent. By~\eqref{eq:pi-expectation-proof}, we have 
    \begin{equation*}
        \BP\lbrack G(a_k, b_k)\rbrack \le \frac{2(32\alpha^{a_k} - 32\alpha^{a_k + 1})}{\alpha^{b_k - 1}/4} \le 2^8 \alpha^{a_k - b_k + 1}, \quad \forall k \in [1, \bullet]_\BZ. 
    \end{equation*}
    Similarly, if $b_0 \le n$, then $\BP\lbrack G_n(b_0)\rbrack \le 2^8 \alpha^{n + 2 - b_0}$. Thus, we conclude that
    \begin{align*}
        \BP\lbrack F_n(\Bc)\rbrack &\le \BP\lbrack G_n(b_0) \cap G(a_1, b_1) \cap \cdots \cap G(a_\bullet, b_\bullet)\rbrack \\
        &= \BP\lbrack G_n(b_0)\rbrack \BP\lbrack G(a_1, b_1)\rbrack \cdots \BP\lbrack G(a_\bullet, b_\bullet)\rbrack \\
        &\le 
        \begin{cases}
            2^{8(\bullet + 1)} \cdot \alpha^{(n + 2 - b_0) + (a_1 - b_1 + 1) + \cdots + (a_\bullet - b_\bullet + 1)} & \text{if } b_0 \le n; \\
            2^{8\bullet} \cdot \alpha^{(a_1 - b_1 + 1) + \cdots + (a_\bullet - b_\bullet + 1)} & \text{if } b_0 = n + 1
        \end{cases} \\
        &\le 2^{8(n + 1)} \cdot \alpha^{\len(\Bc)} \quad \text{(since } \bullet \le n\text{)}. 
    \end{align*}
    One verifies immediately that the assignment
    \begin{equation*}
        \{\text{all possible sequences }\Bc\} \to \{0, 1, 2\}^n \colon \Bc \mapsto \left(j \mapsto 
        \begin{cases}
            0 & \text{if } j \notin [b_0, n]_\BZ \cup \bigcup_{k = 1}^\bullet [b_k, a_k]_\BZ; \\
            1 & \text{if } j = a_k \text{ for some } k \in [1, \bullet]_\BZ; \\
            2 & \text{otherwise}
        \end{cases}
        \right)
    \end{equation*}
    is injective. This implies that the number of all possible sequences $\Bc$ is at most $3^n$. Thus,
    \begin{align*}
        \BE\lbrack\varpi(o, n)^p\rbrack &\le \sum_\Bc \BE\!\left\lbrack\varpi(o, n)^p \mathbf 1_{F_n(\Bc)}\right\rbrack \\
        &\le \sum_\Bc 2^{8(n + 1)} \cdot 65^{pn} \cdot \alpha^{100pn - (99p - 1)\len(\Bc)} \\
        &\le 3^n \cdot 2^{8(n + 1)} \cdot 65^{pn} \cdot \alpha^{(p + 1)n}. 
    \end{align*}
    This completes the proof. 
\end{proof}

\section{Proof of \texorpdfstring{\Cref{thm:main}}{Theorem~\ref{thm:main}}}\label{section:proof}

Let $(\SCT, d_\SCT, \SM_\SCT; o)$ be an infinite Brownian tree. We \emph{claim} that in order to prove \Cref{thm:main}, it suffices to show that $B_1(o; d_\SCT)$ almost surely has conformal dimension $1$. Let $(\SCU, x)$ be the subtree branching off $\llbracket o, \SR_{1/2}(o)\rrbracket$ with the largest diameter. It follows from the PPP construction of $\SCT$ (cf.~\Cref{subsubsection:infinite-Brownian-tree}) that the law of $(\SCU, d_\SCT|_{\SCU \times \SCU}, \SM_\SCT|_\SCU; x)$ and the Brownian tree measure $\SCN$ (cf.~\Cref{subsubsection:Brownian-tree}) are mutually absolutely continuous, which implies that the law of $(\SCU, d_\SCT|_{\SCU \times \SCU}, \SM_\SCT|_\SCU; x)$ on the event $\{\diam(\SCU; d_\SCT) < 1/2\}$ and the restriction of $\SCN$ to the subset of trees with diameter less than $1/2$ are mutually absolutely continuous. Combining this with the scaling property, we obtain that the law of $(\SCU, \SM_\SCT(\SCU)^{-1/2} d_\SCT|_{\SCU \times \SCU}, \SM_\SCT(\SCU)^{-1} \SM_\SCT|_\SCU; x)$ on the event $\{\diam(\SCU; d_\SCT) < 1/2\}$ and the law of a Brownian tree with unit volume are mutually absolutely continuous. On the other hand, on the event $\{\diam(\SCU; d_\SCT) < 1/2\}$, we have $\SCU \subset B_1(o; d_\SCT)$, which implies that the conformal dimension of $\SCU$ is almost surely $1$. Combining the above discussion, we obtain that the conformal dimension of a Brownian tree with unit volume is almost surely $1$. This completes the proof of the \emph{claim}.

It remains to prove that $B_1(o; d_\SCT)$ almost surely has conformal dimension $1$. Let $\sigma$, $\varrho$, $\varsigma$, and $\varpi$ be as in \Cref{def:weight,def:rho-sigma-pi}. Since condition~\eqref{eq:admissibility} is satisfied (cf.~\Cref{lem:admissibility}), by \Cref{subsubsection:hyperbolic-filling}, it suffices to show that for each $p \in (1, 2)$, we may choose the parameter $\alpha$ to be sufficiently small so that whenever $\{o\} = A_0 \subset A_1 \subset A_2 \subset \cdots \subset B_1(o; d_\SCT)$ is an increasing sequence of finite subsets such that each $A_n$ is maximal $\alpha^n$-separated, we have
\begin{equation}\label{eq:main-proof}
    \sum_{(x, n) \in V_n} \prod_{j = 1}^n \varrho(x, j)^p \to 0 \quad \text{as } n \to \infty. 
\end{equation}

Fix $p \in (1, 2)$. Fix $\alpha \in (0, 1)$ to be chosen later. Set $\eta \defeq \alpha^{100}$. Let $\mathop{\mathrm{proj}} \colon \BR \to \SCT$ be the canonical projection mapping. Fix $T > 0$. Let $\{x_k\}_{k \in \BN}$ be conditionally independent samples from $\SM_\SCT|_{\mathop{\mathrm{proj}}([-T, T])}$ (renormalized to be a probability measure). On the event that $B_1(o; d_\SCT) \subset \mathop{\mathrm{proj}}([-T, T])$, for each $n \in \BN$, write
\begin{equation*}
    k_n \defeq \inf\!\left\{K \in \BN : B_1(o; d_\SCT) \subset \bigcup\{B_{\alpha^n/2}(x_k; d_\SCT) : k \in [1, K]_\BZ, \ x_k \in B_1(o; d_\SCT)\}\right\}. 
\end{equation*}
By definition, for each $x \in A_n$, there exists $k \in [1, k_n]_\BZ$ such that $d_\SCT(x, x_k) < \alpha^n/2$. Moreover, since $A_n$ is $\alpha^n$-separated, this assignment $A_n \to [1, k_n]_\BZ \colon x \mapsto k$ is injective. Combining this with \Cref{lem:rho-upper-bound}, we obtain that
\begin{equation*}
    \sum_{(x, n) \in V_n} \prod_{j = 1}^n \varrho(x, j)^p \le \sum_{k \in [1, k_n]_\BZ} \varpi(x_k, n)^p. 
\end{equation*}
Set $\zeta \defeq (p - 1)/2$. By \Cref{lem:net}, almost surely on the event that $B_1(o; d_\SCT) \subset \mathop{\mathrm{proj}}([-T, T])$, we have $k_n \le \alpha^{-(2 + \zeta)n}$ for all sufficiently large $n \in \BN$. Fix $n_\ast \in \BN$. Write $E_{T,n_\ast}$ for the event that $B_1(o; d_\SCT) \subset \mathop{\mathrm{proj}}([-T, T])$ and that $k_n \le \alpha^{-(2 + \zeta)n} \text{ for all } n \ge n_\ast$. Recall from the re-rooting property of $\SCT$ (cf.~\Cref{subsubsection:infinite-Brownian-tree}) that each $(\SCT, d_\SCT, \SM_\SCT; x_k)$ has the same law as $(\SCT, d_\SCT, \SM_\SCT; o)$. Thus, for each $n \ge n_\ast$, 
\begin{align*}
    \BE\!\left\lbrack\left(\sum_{(x, n) \in V_n} \prod_{j = 1}^n \varrho(x, j)^p\right)\mathbf 1_{E_{T,n_\ast}}\right\rbrack &\le \BE\!\left\lbrack\left(\sum_{k \in [1, k_n]_\BZ} \varpi(x_k, n)^p\right)\mathbf 1_{E_{T,n_\ast}}\right\rbrack \\
    &\le \alpha^{-(2 + \zeta)n} \BE\lbrack\varpi(o, n)^p\rbrack \\
    &\le C^n \alpha^{(p - 1)n/2}, 
\end{align*}
where the last inequality follows from \Cref{lem:pi-expectation}. We may choose $\alpha$ to be sufficiently small so that $C \alpha^{(p - 1)/2} < 1$. Then it follows from the Borel--Cantelli lemma that~\eqref{eq:main-proof} holds almost surely on the event $E_{T,n_\ast}$. Since there almost surely exist $T > 0$ and $n_\ast \in \BN$ such that $E_{T,n_\ast}$ occurs, this completes the proof of \Cref{thm:main}. \qed

\bibliographystyle{alpha}
\bibliography{references}

\end{document}